\documentclass[12pt]{article}
\usepackage{amssymb}
\usepackage{amsfonts}
\usepackage{graphicx}
\usepackage{amsmath}

\numberwithin{equation}{section}

\def\dint{\int}
\def\dbigcup{\bigcup}
\def\dbigcap{\bigcap}
\newtheorem{theorem}{Theorem}[section]

\newtheorem{definition}[theorem]{Definition}

\newtheorem{lemma}[theorem]{Lemma}

\newenvironment{proof}[1][Proof]{\textbf{#1.} }{\ \rule{0.5em}{0.5em}}
\begin{document}

\title{Stratonovich's Signatures of Brownian Motion Determine Brownian
Sample Paths}
\date{}
\author{\emph{{\small {\textsc{By Yves LeJan \ and \ Zhongmin Qian}}}} \\
\emph{{\small {Universit\'{e} Paris-Sud and University of Oxford}}}}
\maketitle

\leftskip1truecm \rightskip1truecm \noindent {\textbf{Abstract.}} The
signature of Brownian motion in $\mathbb{R}^{d}$ over a running time
interval $[0,T]$ is the collection of all iterated Stratonovich path
integrals along the Brownian motion. We show that, in dimension $d\geq 2$,
almost all Brownian motion sample paths (running up to time $T$) are
determined by its signature over $[0,T]$.\footnote[1]{%
} 
\footnote[2]{%
}

\leftskip0truecm \rightskip0truecm

\section{Introduction}

Let $W=(W_{t}^{1},\cdots ,W_{t}^{d})_{t\geq 0}$ be a Brownian motion in the
Euclidean space of dimension $d\geq 2$. The Stratonovich signature of $W$
over the duration from time $0$ to time $T$, according to K. T. Chen \cite%
{MR0454968} and T. Lyons \cite{MR2630037}, is the formal series with $d$
indeterminates $X_{1},\cdots ,X_{d}$ whose coefficients are iterated
Stratonovich's path integrals of Brownian sample paths:%
\begin{equation}
S(W)_{[0,T]}=\sum_{n=0}^{\infty }\sum_{\pi \in S_{n}}[\pi _{1}\cdots \pi
_{n}]_{0,T}X_{\pi _{1}}\cdots X_{\pi _{n}}  \label{sy-01}
\end{equation}%
where $S_{n}$ denotes the permutation group of $\{1,\cdots ,n\}$, $\sum_{\pi
\in S_{n}}$ runs through permutations $\pi =(\pi _{1},\cdots ,\pi _{n})\in
S_{n}$, and the square bracket $[\pi _{1}\cdots \pi _{n}]_{s,t}$ denotes the
multiple Stratonovich integral of Brownian motion over $[s,t]$, i.e.%
\begin{equation}
\lbrack \pi _{1}\cdots \pi _{n}]_{s,t}=\dint\limits_{s<t_{1}<\cdots
<t_{k}<t}\circ dW_{t_{1}}^{\pi _{1}}\circ \cdots \circ dW_{t_{n}}^{\pi _{n}}%
\text{.}  \label{sy-02}
\end{equation}%
These integrals may be defined by means of It\^{o}'s integration. In fact,
multiple integrals may be defined inductively by%
\begin{equation*}
\lbrack \pi _{1}\cdots \pi _{n}]_{s,t}=\int_{s}^{t}[\pi _{1}\cdots \pi
_{n-1}]_{[s,r]}\circ dW_{r}^{\pi _{n}}
\end{equation*}%
where $\circ d$ indicates the integration in Stratonovich's sense, which in
turn can be expressed in terms of It\^{o}'s and ordinary integrals.

If one is not concerned about underlying algebraic structures defined by
iterated integrals, it is not necessary to approach the Stratonovich
signature through the formal series (\ref{sy-01}). We consider the
collection of all possible iterated Stratonovich integrals $[\pi _{1}\cdots
\pi _{n}]_{0,T}$, emphasizing the fact that they are all taken over a fixed
time interval $[0,T]$, as the Stratonovich signature(s) of Brownian motion
(over $[0,T]$). Since we will work on signatures over a fixed interval, the
lower script $0$ and $T$ will be omitted if no confusion may arise, for the
sake of simplicity of notations. Without losing generality we may from now
on assume that $T=1$.

Since the notion of signatures is so significant in this paper, we thus
would like to present a formal definition.

\begin{definition}
\label{def-01}Let $W=(W_{t})_{t\geq 0}$ be a Brownian motion in $\mathbb{R}%
^{d}$ starting at $0$. Then the Stratonovich signature (or signatures) of $W$
over $[0,1]$ is the collection of all iterated Stratonovich integrals 
\begin{equation*}
\lbrack j_{1}\cdots j_{n}]=\dint\limits_{0<t_{1}<\cdots <t_{k}<1}\circ
dW_{t_{1}}^{j_{1}}\circ \cdots \circ dW_{t_{n}}^{j_{n}}
\end{equation*}%
where $n$ runs through $1,2,\cdots ,$ and $j_{1},\cdots ,j_{n}\in \{1,\cdots
,d\}$.
\end{definition}

The interest for signatures of paths has a long history. First of all,
sequences of multiple iterated integrals arise naturally in Picard's
iteration of solving ordinary differential equations. Multiple iterated
integrals of Brownian motion appeared already in early 1930's in Wiener's
celebrated work on harmonic analysis on the Wiener space, and K. It\^{o}
studied them in terms of his integration theory. Meanwhile, from 1950's to
late1970's, in a series of articles \cite{MR525704}, \cite{MR0454968}, \cite%
{MR0085251}, \cite{Chen1} etc. K. T. Chen demonstrated the usefulness of
iterated integrals along piecewise smooth paths in manifolds. K. T. Chen
showed the interesting algebraic structures defined by sequences of iterated
integrals, developed a representation theory, and established a homotopy
theory in terms of iterated integrals. The importance of multiple
Stratonovich integrals, however, surprisingly was not recognized until the
important contributions by Wong-Zakai \cite{MR0183023}, Ikeda and Watanabe 
\cite{MR637061}, in which the convergence theorem for solutions to
stochastic differential equations in Stratonovich's sense was proved. The
definite role played by iterated Stratonovich's integrals was finally
revealed in T. Lyons \cite{MR1654527} (also see \cite{MR2036784}) in which a
universal limit theorem for solutions of Stratonovich's stochastic
differential equations was proved. T. Lyons has realized that the key
elements for defining an integration theory along a continuous path which is
not necessary piecewise smooth is the sequence of iterated integrals that
must be specified. This idea led to the discovery of the $p$-variation
metric among continuous paths with finite variations, which allows to
develop the theory of rough paths.

It has been conjectured that the signature of a path over a fixed time
duration $[0,1]$, which can be read out at the terminal time $1$, should be
a good summary of information about the flow of timely ordered events,
recorded in its path during time $0$ to time $1$. K. T. Chen \cite{Chen1}
first proved that indeed it is possible to recover the whole path (up to
tree-like components of the path which are not counted in its signature) by
reading its signature. B. Hambly and T. Lyons \cite{MR2630037} extended and
quantified Chen's result to rectifiable curves in multi-dimensional spaces.
Unfortunately, these results are not applicable to interesting random
curves, since, for example, almost all sample paths of a non-trivial
diffusion process are not rectifiable.

In this article, we demonstrate that for $d\geq 2$ almost all $d$%
-dimensional Brownian paths can be recovered from its Stratonovich's
signature. In other words, theoretically, all information recorded in
Brownian motion from $0$ to $1$ can be read out from the Stratonovich
signature over $[0,1]$.

To state our main result more precisely, we need to introduce more
notations. Let $\mathcal{F}_{t}^{0}=\sigma \{W_{s}:s\leq t\}$ be the
filtration generated by $W$, and $\mathcal{F}_{1}$ be the completion of $%
\mathcal{F}_{1}^{0}$ (under the Brownian measure $P$), and $\mathcal{G}_{1}$
be the complete $\sigma $-algebra generated by the Stratonovich signatures,
i.e. the completion of the $\sigma $-algebra $\sigma \{[\pi _{1}\cdots \pi
_{n}]_{0,1}:\pi \in S_{n}$; $n\in \mathbb{N}\}$.

Our main result may be stated as follows

\begin{theorem}
\label{th-main} $\mathcal{F}_{1}=\mathcal{G}_{1}$. Therefore the
Stratonovich signature determines Brownian sample paths almost surely.
\end{theorem}

To prove this theorem, we need to develop a method of reconstructing almost
all Brownian sample paths given their signatures. We will come to this point
shortly.

In order to appreciate why Stratonovich signatures are able to represent the
sample paths of Brownian motion, let us look at how to obtain iterated
integrals of smooth differential forms along Brownian motion paths in terms
of the Stratonovich signatures. The remarkable fact, which certainly goes
back to K. T. Chen \cite{Chen1} for the deterministic case, is that any
polynomials of Brownian motion (evaluated at a fixed time $1$) is a linear
combination of the signatures over $[0,1]$. In fact%
\begin{equation}
W_{t}^{j_{1}}\cdots W_{t}^{j_{n}}=\sum_{\pi \in S_{n}}[j_{\pi _{1}}\cdots
j_{\pi _{n}}]_{0,t}\text{.}  \label{poly-01}
\end{equation}%
This formula can be proved by integrating by parts:%
\begin{equation*}
W_{t}^{j_{1}}W_{t}^{j_{2}}=\left[ j_{1}j_{2}\right] _{0,t}+\left[ j_{2}j_{1}%
\right] _{0,t}
\end{equation*}%
and for $n\geq 2$ 
\begin{eqnarray*}
W_{t}^{j_{1}}\cdots W_{t}^{j_{n}}W_{t}^{j_{n+1}} &=&\sum_{\pi \in
S_{n}}\int_{0}^{t}\left[ j_{\pi _{1}}\cdots j_{\pi _{n}}\right] _{0,s}\circ
dW_{s}^{j_{n+1}} \\
&&+\sum_{\pi \in S_{n}}\int_{0}^{t}W_{s}^{j_{n+1}}\circ d\left[ j_{\pi
_{1}}\cdots j_{\pi _{n}}\right] _{0,s} \\
&=&\sum_{\pi \in S_{n}}\left[ j_{\pi _{1}}\cdots j_{\pi _{n}}j_{n+1}\right]
_{0,t} \\
&&+\sum_{\pi \in S_{n}}\int_{0}^{t}W_{s}^{j_{n+1}}\left[ j_{\pi _{1}}\cdots
j_{\pi _{n-1}}\right] _{s}\circ dW_{s}^{j_{\pi _{n}}}
\end{eqnarray*}%
and (\ref{poly-01}) follows. If $\alpha ^{1}$, $\cdots $, $\alpha ^{k}$ are
smooth differential forms on $\mathbb{R}^{d}$ with compact supports, then
iterated Stratonovich integrals $[\alpha ^{1}\cdots \alpha ^{k}]_{s,t}$ are
defined inductively by 
\begin{equation*}
\lbrack \alpha ^{1}\cdots \alpha ^{k}]_{s,t}=\int_{s}^{t}[\alpha ^{1}\cdots
\alpha ^{k-1}]_{s,u}\alpha ^{k}(\circ dW_{u})\text{.}
\end{equation*}%
Since polynomials are dense in $C^{k}$ functions for any $k$ under uniform
convergence over compact subsets, therefore all iterated Stratonovich
integrals of $1$-forms against $W$ are measurable functionals of the
signatures. This is the context of the following lemma.

\begin{lemma}
\label{lem-t1}If $\alpha ^{1}$, $\cdots $, $\alpha ^{k}$ are smooth
differential forms on $\mathbb{R}^{d}$ with compact supports, then $[\alpha
^{1}\cdots \alpha ^{k}]_{0,1}$ is $\mathcal{G}_{1}$-measurable.
\end{lemma}

\begin{proof}
If $\alpha ^{l}$ have polynomial coefficients, then we have seen that $%
[\alpha ^{1}\cdots \alpha ^{k}]_{0,1}$ is a linear combination of the
Stratonovich signatures, so it is $\mathcal{G}_{1}$-measurable. In general
case, we may approximate $\alpha ^{1}$, $\cdots $, $\alpha ^{k}$ by
polynomials $\alpha _{n}^{1}$, $\cdots $, $\alpha _{n}^{k}$ in $C^{k+1}$
norm, so that%
\begin{equation*}
\lbrack \alpha _{n}^{1}\cdots \alpha _{n}^{k}]_{s,t}\rightarrow \lbrack
\alpha ^{1}\cdots \alpha ^{k}]_{s,t}
\end{equation*}%
in $L^{2}(\Omega ,\mathcal{F},P)$. This yields that $[\alpha ^{1}\cdots
\alpha ^{k}]_{0,T}$ is $\mathcal{G}_{T}$-measurable.
\end{proof}

These iterated Stratonovich integrals $[\alpha ^{1}\cdots \alpha ^{k}]_{0,1}$
may be considered as "extended\textquotedblright\ signatures of $W$ over $%
[0,1]$.

Since there is no essential differences in our proof of Theorem \ref{th-main}
between dimension two and the higher dimensional case, we therefore
concentrate on the case $d=2$. The main idea and the key steps in the proof
of Theorem \ref{th-main} are described as follows.

To construct approximations of Brownian motion $W$ in terms of a countable
family of extended signatures, for each $\varepsilon >0$ we construct an $%
\varepsilon $-grid so that $\mathbb{R}^{2}$ is divided into squares with
center at $\boldsymbol{z}\varepsilon =(z_{1}\varepsilon ,z_{2}\varepsilon )$
and wide $\varepsilon $, and let 
\begin{equation*}
S_{\boldsymbol{z}}=\{(x_{1},x_{2}):|x_{1}-z_{1}\varepsilon
|+|x_{2}-z_{2}\varepsilon |\leq \frac{1}{2}\varepsilon (1-\varepsilon )\}
\end{equation*}%
which is strictly located inside the squares with the same center. We
naturally construct an approximation by polygons which join the centers of
the squares $S_{\boldsymbol{z}}$ which have been visited by the Brownian
motion paths $W$. It is not very difficult to show these polygons converge
to Brownian motion paths almost surely, and we want to show that these
polygonal approximations are indeed determined by the Stratonovich
signatures of $W$. To this end, we construct a smooth differential $1$-form $%
\phi ^{\boldsymbol{z}}$ which has a compact support inside the squares $S_{%
\boldsymbol{z}}$ so that for different indices $\boldsymbol{z}\in \mathbb{Z}%
^{2}$, these differential 1-forms $\phi ^{\boldsymbol{z}}$ have disjoint
supports. The key observation is that the Stratonovich integral $\int \phi ^{%
\boldsymbol{z}}(\circ dW)$ does not vanish almost surely over the duration
that the Brownian motion has visited $S_{\boldsymbol{z}}$. This crucial fact
allows us to identify those squares the Brownian motion has visited entirely
in terms of the signatures of the Brownian motion.

\section{Several technical facts}

In this section we establish several technical facts which will be used in
the proof of Theorem \ref{th-main}.

A planer square is a nice domain but its boundary has four corners and thus
is not $C^{1}$. For the technical reasons we consider a domain obtained from
a square by replacing the portion of the boundary near each corner by a
quarter of small circles. More precisely, for a small $\frac{1}{4}%
>\varepsilon >0$, and, as we will use this parameter $\varepsilon $ for
other constructions, for $\beta \gg 1$, let 
\begin{equation*}
D=\left\{ (x_{1},x_{2}):0\leq x_{1},x_{2}\leq \frac{1}{2}\right\} \setminus
\left\{ \left\vert x_{1}-\frac{1}{2}+\varepsilon ^{\beta }\right\vert
^{2}+\left\vert x_{1}-\frac{1}{2}+\varepsilon ^{\beta }\right\vert ^{2}\geq
\varepsilon ^{2\beta }\right\}
\end{equation*}%
and the typical planer domain we will handle is%
\begin{equation}
G=\{(x_{1},x_{2}):(|x_{1}|,|x_{2}|)\in D\}\text{.}  \label{domaing}
\end{equation}

For $a>0$, $G_{a}$ denotes the similar planer domain $aG$, i.e. $%
G_{a}=\{x=(x_{1},x_{2}):(ax_{1},ax_{2})\in G\}$.

Let $W_{t}=(W_{t}^{1},W_{t}^{2})$ be a two dimensional Brownian motion on a
probability space $(\Omega ,\mathcal{F},P)$, and let $a_{3}>a_{2}>a_{1}$. Let%
\begin{equation*}
S_{0}=\inf \{t>0:W_{t}\in \partial G_{a_{3}}\}\text{,}
\end{equation*}%
\begin{equation*}
S_{1}=\inf \{t>S_{0}:W_{t}\in \partial G_{a_{1}}\}
\end{equation*}%
and%
\begin{equation*}
S_{2}=\inf \{t>S_{1}:W_{t}\in \partial G_{a_{2}}\}
\end{equation*}%
which are stopping times, finite almost surely. We are interested in the
distribution of the random variable $X=\int_{S_{0}}^{\tau }\phi (\circ
dW_{s})$, where $\phi $ is a differential $1$-form which coincides with $%
x^{2}dx^{1}$ on $G_{a_{2}}$, conditional to $\{S_{1}<\tau \}$.

To this end, consider the diffusion process $X=(X^{1},X^{2},X^{3})$ in $%
\mathbb{R}^{3}$ associated with the following stochastic differential
equations%
\begin{equation}
dX_{t}^{1}=dW_{t}^{1}\text{; \ }dX_{t}^{2}=dW_{t}^{2}\text{; \ }%
dX_{t}^{3}=X_{t}^{2}\circ dW_{t}^{1}\text{.}  \label{sde-r1}
\end{equation}%
It is an easy exercise to calculate the infinitesimal generator of $X$,
which is $L=\frac{1}{2}\left( A_{1}^{2}+A_{2}^{2}\right) $, where $A_{1}=%
\frac{\partial }{\partial x_{1}}+x_{2}\frac{\partial }{\partial x_{3}}$ and $%
A_{2}=\frac{\partial }{\partial x_{2}}$. In particular, the Lie bracket $%
[A_{1},A_{2}]=-\frac{\partial }{\partial x^{3}}$, so that $L$ is
hypoelliptic (Theorem 1.1, page 149, H\"{o}mander \cite{MR0222474}).

\begin{lemma}
\label{lem-m1}Let $W$ be Brownian motion in $\mathbb{R}^{2}$ on $(\Omega ,%
\mathcal{F},P)$ started from a point at $\partial G_{a_{1}}$, $S=\inf
\{t>0:W_{t}\in \partial G_{a_{2}}\}$, and $\xi =$ $\int_{0}^{S}W_{s}^{2}%
\circ dW_{s}^{1}$. Then, for any $y\in \partial G_{a_{2}}$, the conditional
distribution $P\{\xi \in dz|W_{S}=y\}$ has a continuous density function in $%
z$.
\end{lemma}

\begin{proof}
Let $D=G_{a_{2}}\times \mathbb{R}^{1}$, and $S=\inf \{t\geq 0:X_{t}\notin
D\} $ the first exit time of the diffusion process $X$. Then, $D$ has a $%
C^{1}$-boundary (this is the reason for which we use rounded squares) and
the condition required in \cite{MR738578} is satisfied, as the normal to the
boundary belongs to the plane spanned by $A_1$ and $A_2$.  Thus, according
to a theorem of Ben Arous, Kusuoka and Stroock (Theorem 1.22, page 181, in 
\cite{MR738578}), the Poisson measure of $L$ on the open domain $D$ has a
(smooth) density, which implies that the distribution of $X_{S}$ has a
continuous density function on $\partial D$ with respect to the Lebesgue
measure on $\partial D$. Therefore the conditional distribution $P\{\xi \in
dz|W_{S}=y\}$ has a continuous density on $\mathbb{R}^{1}$ for $y\in
\partial G_{a_{2}}$.
\end{proof}

Let $f(x_{1},x_{2})$ be a smooth function on $\mathbb{R}^{2}$ with a support
in $G_{a_{3}}$ such that $f(x_{1},x_{2})=x_{2}$ on $G_{a_{2}}$. Consider the
smooth differential 1-form $\phi =f(x_{1},x_{2})dx_{1}$ on $\mathbb{R}^{2}$.

\begin{lemma}
Under above assumptions and notations. Let $Z=\int_{S_{1}}^{S_{2}}\phi
(\circ W_{S})$. Then the conditional distribution of $Z$ given $%
W_{S_{1}}=(x_{1},x_{2})$ and $W_{S_{2}}=(y_{1},y_{2})$ has a continuous
density function, i.e.%
\begin{equation}
P\{Z\in
dz|W_{S_{1}}=(x_{1},x_{2}),W_{S_{2}}=(y_{1},y_{2})%
\}=p((x_{1},x_{2}),(y_{1},y_{2}),z)dz  \label{des-01}
\end{equation}%
for some nonnegative function $p$.
\end{lemma}

\begin{proof}
This follows from the Strong Markov property of $X$ and the previous Lemma.
\end{proof}

\begin{lemma}
\label{lem-key1}Under conditions and notations described above. \ Let $U$ be
an open subset such that $\overline{G_{a_{3}}}\cap U=\emptyset $ and $\tau
=\inf \{t>S_{0}:W_{t}\in \partial U\}$ be a hitting time. Let $T=S_{2}+\tau
\circ S_{2}$. Then the random variable $\eta =\int_{S_{0}}^{T}\phi (\circ
dW_{s})\neq 0$ almost surely on $\{S_{1}<T\}$.
\end{lemma}

\begin{proof}
Write%
\begin{equation*}
\eta =\int_{S_{1}}^{S_{2}}\phi (\circ dW_{s})+\int_{S_{0}}^{S_{1}}\phi
(\circ dW_{s})+\int_{S_{2}}^{T}\phi (\circ dW_{s})\text{.}
\end{equation*}%
For any stopping time $S$ we have two $\sigma $-fields, namely $\mathcal{F}%
_{S}$ which is the $\sigma $-algebra of events happening before $S$, and $%
\mathcal{F}_{>S}$ the $\sigma $-algebra of events depending on the path
after stopping time $S$. By definition, $1_{\{S_{1}<T\}}\int_{S_{0}}^{S_{1}}%
\phi (\circ dW_{s})$ is $\mathcal{F}_{S_{1}}$-measurable and $%
1_{\{S_{1}<T\}}\int_{S_{2}}^{T}\phi (\circ dW_{s})$ is $\mathcal{F}%
_{S_{1}}\vee \mathcal{F}_{>S_{2}}$ measurable. Let $Y=\int_{S_{0}}^{S_{1}}%
\phi (\circ dW_{s})+\int_{S_{2}}^{T}\phi (\circ dW_{s})$ for simplicity. By
the strong Markov property%
\begin{eqnarray*}
E\left\{ 1_{\{S_{1}<T\}}\int_{S_{1}}^{S_{2}}\phi (\circ dW_{s})|\mathcal{F}%
_{S_{1}}\vee \mathcal{F}_{>S_{2}}\right\} &=&1_{\{S_{1}<T\}}E\left\{
\int_{S_{1}}^{S_{2}}\phi (\circ dW_{s})|\mathcal{F}_{S_{1}}\vee \mathcal{F}%
_{>S_{2}}\right\} \\
&=&1_{\{S_{1}<T\}}E\left\{ \int_{S_{1}}^{S_{2}}\phi (\circ
dW_{s})|W_{S_{1}},W_{S_{2}}\right\}
\end{eqnarray*}%
so that%
\begin{equation*}
E\left\{ F(\eta )1_{\{S_{1}<T\}}|\mathcal{F}_{S_{1}}\vee \mathcal{F}%
_{>S_{2}}\right\} =1_{\{S_{1}<T\}}E\left\{ F\left( \int_{S_{1}}^{S_{2}}\phi
(\circ dW_{s})+Y\right) |\mathcal{F}_{S_{1}}\vee \mathcal{F}%
_{>S_{2}}\right\} \text{.}
\end{equation*}%
Suppose $F(z+y)=\sum_{j}H_{j}(z)K_{j}(y)$, then%
\begin{eqnarray*}
E\left\{ F(\eta )|\mathcal{F}_{S_{1}}\vee \mathcal{F}_{>S_{2}}\right\}
&=&\sum_{j}K_{j}(Y)E\left\{ H_{j}\left( \int_{S_{1}}^{S_{2}}\phi (\circ
dW_{s})\right) |\mathcal{F}_{S_{1}}\vee \mathcal{F}_{>S_{2}}\right\} \\
&=&\sum_{j}K_{j}(Y)E\left\{ H_{j}\left( \int_{S_{1}}^{S_{2}}\phi (\circ
dW_{s})\right) |W_{S_{1}},W_{S_{2}}\right\}
\end{eqnarray*}%
and therefore%
\begin{equation*}
E\left\{ F(\eta )1_{\{S_{1}<T\}}\right\} =\sum_{j}E\left\{ K_{j}(Y)E\left[
H_{j}\left( \int_{S_{1}}^{S_{2}}\phi (\circ dW_{s})\right)
|W_{S_{1}},W_{S_{2}}\right] 1_{\{S_{1}<T\}}\right\} \text{.}
\end{equation*}%
Since $\int_{S_{1}}^{S_{2}}\phi (\circ dW_{s})$ has a conditional
probability density $p(x,y,z)$%
\begin{equation*}
E\left[ 1_{\{S_{1}<T\}}\int_{S_{1}}^{S_{2}}\phi (\circ dW_{s})\in
dz|W_{S_{1}}=x,W_{S_{2}}=y\right] =p(x,y,z)dz
\end{equation*}%
and thus%
\begin{eqnarray*}
E\left\{ F(\eta )1_{\{S_{1}<T\}}\right\} &=&E\left\{ 1_{\{S_{1}<T\}}\int_{%
\mathbb{R}}\sum_{j}K_{j}(Y)H_{j}\left( z\right)
p(W_{S_{1}},W_{S_{2}},z)dz\right\} \\
&=&E\left\{ 1_{\{S_{1}<T\}}\int_{\mathbb{R}}F(Y+z)p(W_{S_{1}},W_{S_{2}},z)dz%
\right\} \text{.}
\end{eqnarray*}%
In particular $P\{\eta =0,S_{1}<T\}=0$.
\end{proof}

\section{Constructing approximations to Brownian paths}

In this section, we construct polygonal approximations to the planer
Brownian motion sample paths by tracing the sample paths of Brownian motion
through prescribed $\varepsilon $-grids laid out in the plane. Our
construction equally applies to higher dimensional Brownian motion with only
minor modifications which we will leave to the reader.

To make our arguments clear, let us work with the classical Wiener space $(%
\boldsymbol{W},\mathcal{B},P)$, where $\boldsymbol{W}$ is the space of all
continuous paths in $\mathbb{R}^{2}$ started at $0$, $\mathcal{B}$ is the
Borel $\sigma $-algebra on $\boldsymbol{W}$ and $P$ is the unique
probability so that the coordinate process $W=(W^{1},W^{2})$ is a planer
Brownian motion on $(\boldsymbol{W},\mathcal{B},P)$ started at $0$.

Let $\varepsilon \in (0,\frac{1}{4})$. Recall that $G$ is the planer domain
defined by (\ref{domaing}) which is the planer square with corners rounded.
For $\boldsymbol{z}=(z_{1},z_{2})\in \mathbb{Z}^{2}$ we assign three boxes $%
H_{\boldsymbol{z}}^{\varepsilon }\subset K_{\boldsymbol{z}}^{\varepsilon
}\subset Z_{\boldsymbol{z}}^{\varepsilon }$ which are all similar domains to 
$G$, with a common center $\varepsilon \boldsymbol{z}$ lies on the $%
\varepsilon $-lattice $\varepsilon \mathbb{Z}$: 
\begin{equation*}
H_{\boldsymbol{z}}^{\varepsilon }=\varepsilon \boldsymbol{z}+\varepsilon
(1-\varepsilon )G\text{,}
\end{equation*}%
\begin{equation*}
K_{\boldsymbol{z}}^{\varepsilon }=\varepsilon \boldsymbol{z}+\varepsilon
\left( 1-\varepsilon +\frac{\varepsilon \varphi (\varepsilon )}{2}\right) G%
\text{,}
\end{equation*}%
\begin{equation*}
Z_{\boldsymbol{z}}^{\varepsilon }=\varepsilon \boldsymbol{z}+\varepsilon
\left( 1-\varepsilon +\varepsilon \varphi (\varepsilon )\right) G\text{,\ }
\end{equation*}%
and%
\begin{equation*}
V_{\boldsymbol{z}}^{\varepsilon }=\varepsilon \boldsymbol{z}+\varepsilon G
\end{equation*}%
where $\varphi (\varepsilon )\ll \varepsilon ^{\alpha }$ (with $\alpha \geq
10$) but to be chosen late on.

Let us notice that the gap between $Z_{\boldsymbol{z}}^{\varepsilon }$ and
the box $V_{\boldsymbol{z}}^{\varepsilon }$ has a magnitude $\varepsilon
^{2}(1-\varphi (\varepsilon ))$, while the magnitude of the gap between $H_{%
\boldsymbol{z}}^{\varepsilon }$ and $K_{\boldsymbol{z}}^{\varepsilon }$ is $%
\frac{1}{2}\varepsilon ^{2}\varphi (\varepsilon )$. Since $\varphi
(\varepsilon )\ll \varepsilon ^{\alpha }$ so that 
\begin{equation*}
\varepsilon ^{2}(1-\varphi (\varepsilon ))\gg \frac{1}{2}\varepsilon
^{2}\varphi (\varepsilon )
\end{equation*}%
a crucial fact we will use below.

If $A\subset \mathbb{R}^{2}$, then $T_{A}$ denotes the hitting time of $A$
by the Brownian motion $W$.

\begin{lemma}
\label{lemc1}There is $\varphi (\varepsilon )\ll \varepsilon ^{\alpha }$
(with $\alpha \geq 11$) and $\beta \gg 10$ such that for every $\boldsymbol{z%
}=(z_{1},z_{2})\in \mathbb{Z}^{2}$ and $x\in \partial Z_{\boldsymbol{z}%
}^{\varepsilon }$ 
\begin{equation}
P\{T_{\partial V_{\boldsymbol{z}}^{\varepsilon }}<T_{H_{\boldsymbol{z}%
}^{\varepsilon }}|W_{0}=x\}\leq \varepsilon ^{10}\text{.}  \label{cap-02}
\end{equation}
\end{lemma}

\begin{proof}
We need to show that the probability on the left-hand side is dominated by
the ratio of the distances between $x$ to $\partial H_{\boldsymbol{z}%
}^{\varepsilon }$ and to $\partial V_{\boldsymbol{z}}^{\varepsilon }$ which
is $\frac{\varphi (\varepsilon )}{1-\varphi (\varepsilon )}$, which in turn
yields the bound in (\ref{cap-02}) as $\varepsilon <\frac{1}{4}$ by
increasing $\alpha $ to kill any possible constant appearing in the
domination. This is standard for one dimensional Brownian motion. Similar
estimates may be obtained by means of potential theory. Clearly the
left-hand side of (\ref{cap-02}) does not depend on $\boldsymbol{z}\in 
\mathbb{Z}^{2}$ so let us assume $\boldsymbol{z}=0$. Let $u$ be the unique
harmonic function on $V_{\boldsymbol{z}}^{\varepsilon }\setminus H_{%
\boldsymbol{z}}^{\varepsilon }$ such that $u=1$ on $\partial V_{\boldsymbol{z%
}}^{\varepsilon }$ and $u=0$ on $\partial H_{\boldsymbol{z}}^{\varepsilon }$%
. Then, $u(W_{t\wedge T_{\partial V_{\boldsymbol{z}}^{\varepsilon }}\wedge
T_{H_{\boldsymbol{z}}^{\varepsilon }}})$ is a bounded \ martingale, so that 
\begin{equation*}
u(x)=P\{T_{\partial V_{\boldsymbol{z}}^{\varepsilon }}<T_{H_{\boldsymbol{z}%
}^{\varepsilon }}|W_{0}=x\}\text{.}
\end{equation*}%
By the uniform continuity of the potential $u$ with respect to the distance
of $x$ to the interior boundary $\partial H_{\boldsymbol{z}}^{\varepsilon }$
(for example see sections 4-2 in Port and\ Stone \cite{MR0492329}), we may
chose $\varphi (\varepsilon )$ small enough so that $x$ is closer to $%
\partial H_{\boldsymbol{z}}^{\varepsilon }$ than to $\partial V_{\boldsymbol{%
z}}^{\varepsilon }$, to ensure that $u(x)\leq \varepsilon ^{10}$ as long as $%
x\in \partial Z_{\boldsymbol{z}}^{\varepsilon }$.

To see the magnitude, we can consider the harmonic function on the disk
centered at $0$ with radius $\varepsilon $%
\begin{equation*}
w(x_{1},x_{2})=\frac{1}{\log \left( 1+2\varepsilon -\varepsilon ^{2}\right) }%
\log \left( \frac{x_{1}^{2}+x_{2}^{2}}{\varepsilon ^{2}}+2\varepsilon
-\varepsilon ^{2}\right)
\end{equation*}%
which vanishes on $\rho \equiv \sqrt{x_{1}^{2}+x_{2}^{2}}=\varepsilon
(1-\varepsilon )$ and is $1$ on $\rho =\varepsilon $. At $\rho =\varepsilon
(1-\varepsilon )+\varepsilon \varphi (\varepsilon )$%
\begin{eqnarray*}
w(x_{1},x_{2}) &=&\frac{1}{\log \left( 1+2\varepsilon -\varepsilon
^{2}\right) }\log \left( 1+2\varphi (\varepsilon )+\varphi (\varepsilon
)^{2}-2\varepsilon \varphi (\varepsilon )\right) \\
&\leq &C\frac{\varphi (\varepsilon )}{\varepsilon }\text{.}
\end{eqnarray*}%
Similar estimates hold for our rounded squares. In dimension $2$, this can
be done by a proper comformal transformation.
\end{proof}

In what follows we choose such $\varphi $ and $\beta $ so that (\ref{cap-02}%
) holds for small $\varepsilon \in (0,1/4)$.

For each path $w\in \boldsymbol{W}$, define a sequence $\{\tau
_{k}(w):k=0,1,2,\cdots \}$ of stopping times which trace the crossings of
the path $w$ through the $\varepsilon $-grid lattice $\varepsilon \mathbb{Z}%
^{2}$. Let $\tau _{0}(w)=0$ and $\boldsymbol{n}_{0}(w)=(0,0)$, and define $%
\tau _{k}(w)$ and $\boldsymbol{n}_{k}(w)$ inductively by 
\begin{equation*}
\tau _{k}(w)=\inf \text{ }\left\{ t>\tau _{k-1}(w):w_{t}\in \dbigcup\limits_{%
\boldsymbol{z}\neq \boldsymbol{n}_{k-1}(w)}H_{\boldsymbol{z}}^{\varepsilon
}\right\}
\end{equation*}%
and $\boldsymbol{n}_{k}(w)\in \mathbb{Z}^{2}$ such that $w(\tau _{k}(w))\in
H_{\boldsymbol{n}_{k}(w)}^{\varepsilon }$ if $\tau _{k}(w)<\infty $, and $%
\boldsymbol{n}_{k+1}(w)=\boldsymbol{n}_{k}(w)$ if $\tau _{k}(w)=\infty $.
Then $\{\tau _{k}:k=0,1,\cdots \}$ is a strictly increasing sequence of
stopping times, and $\tau _{k}\uparrow \infty $ almost surely as $k\uparrow
\infty $.

Let us use $\{\zeta _{k}:k=0,1,\cdots \}$ and $\{\boldsymbol{m}%
_{k}:k=0,1,\cdots \}$ to denote the corresponding sequences obtained in the
previous definition with box $H_{\boldsymbol{z}}^{\varepsilon }$ replaced by 
$Z_{\boldsymbol{z}}^{\varepsilon }$. In other words%
\begin{equation*}
\zeta _{k}(w)=\inf \text{ }\left\{ t>\zeta _{k-1}(w):w_{t}\in
\dbigcup\limits_{\boldsymbol{z}\neq \boldsymbol{m}_{k-1}(w)}Z_{\boldsymbol{z}%
}^{\varepsilon }\right\}
\end{equation*}%
etc.

Let $M_{H}(w)=\inf \{k:\tau _{k+1}(w)>1\}$ and $M_{Z}(w)=\inf \{k:\zeta
_{k+1}(w)>1\}$. Then both $M_{H}<\infty $ and $M_{Z}<\infty $ almost surely.
Since a path which hits the box $H_{\boldsymbol{z}}^{\varepsilon }$ must
first hit the larger one $Z_{\boldsymbol{z}}^{\varepsilon }$ so that $\zeta
_{k}\leq \tau _{k}$ for any $k$, and therefore $M_{H}\leq M_{Z}$. The last
inequality says a continuos path at least hit as many larger boxes than
smaller ones.

Let us construct $w(\varepsilon )$ to be the polygon assuming the point $%
\boldsymbol{n}_{k}\varepsilon $ at time $\tau _{k}$, that is,%
\begin{equation*}
w(\varepsilon )_{t}=\varepsilon \boldsymbol{n}_{k-1}(w)+\frac{t-\tau
_{k-1}(w)}{\tau _{k}(w)-\tau _{k-1}(w)}\varepsilon \boldsymbol{n}_{k+1}(w)%
\text{ \ \ \ if }t\in \lbrack \tau _{k-1}(w),\tau _{k}(w)]
\end{equation*}%
for $l=0,1,\cdots $. We show that $w(\varepsilon )$ converges to the
Brownian curves almost surely as $\varepsilon \downarrow 0$.

\begin{lemma}
\label{lemc2}Let $W=(W_{t})_{t\geq 0}$ be a planer Brownian motion started
at some point inside the box $H_{\boldsymbol{0}}^{\varepsilon }$, and 
\begin{equation*}
\tau =\inf \left\{ t>0:W_{t}\in \dbigcup\limits_{\boldsymbol{z}\neq 
\boldsymbol{0}}H_{\boldsymbol{z}}^{\varepsilon }\right\} \text{ .}
\end{equation*}%
Then%
\begin{equation*}
P\left\{ \sup_{0\leq t\leq \tau }|W_{t}|>3\sqrt{2}\varepsilon \right\} \leq
\left( \frac{1}{3}\right) ^{\left[ \frac{1}{2\varepsilon }\right] }\text{ .}
\end{equation*}
\end{lemma}

\begin{proof}
Let $\boldsymbol{z}=(z_{1},z_{2})\in \mathbb{Z}^{2}$ be the random variable
such that $W_{\tau }\in H_{\boldsymbol{z}}^{\varepsilon }$. If 
\begin{equation*}
\boldsymbol{z}\neq (\pm 1,\pm 1),(\pm 1,0)\text{ or }(0,\pm 1)
\end{equation*}%
or the Brownian motion $W$ runs out off the square $[-3\varepsilon
,3\varepsilon ]\times \lbrack -3\varepsilon ,3\varepsilon ]$, then $W$ must
travel through a narrow strip of wideness $\varepsilon ^{2}$ and length $%
\varepsilon -2\varepsilon ^{\beta }$, so that the probability%
\begin{equation*}
P\left\{ \boldsymbol{z}\neq (\pm 1,\pm 1),(\pm 1,0)\text{ or }(0,\pm
1)\right\} \leq \left( \frac{1}{3}\right) ^{\left[ \frac{1}{2\varepsilon }%
\right] }\text{.}
\end{equation*}%
Therefore%
\begin{equation*}
P\left\{ \sup_{0\leq t\leq \tau }|W_{t}|>3\sqrt{2}\varepsilon \right\} \leq
\left( \frac{1}{3}\right) ^{\left[ \frac{1}{2\varepsilon }\right] }\text{ .}
\end{equation*}
\end{proof}

\begin{lemma}
There is a sequence $\varepsilon _{n}\downarrow 0$, such that%
\begin{equation*}
P\left\{ w:\lim_{n\rightarrow \infty }\inf_{\sigma }\sup_{0\leq t\leq
1}|w_{t}-w(\varepsilon _{n})_{\sigma (t)}|=0\right\} =1
\end{equation*}%
where $\inf_{\sigma }$ takes over all possible parametrization.
\end{lemma}

\begin{proof}
We need to estimate the numbers of the crossings between different $H_{%
\boldsymbol{z}}^{\varepsilon }$ during the time $0$ to $1$. Note that%
\begin{eqnarray*}
P\left\{ M_{H}\geq k\text{ }\right\} &\leq &P\left\{ \text{at least for one }%
l\text{, }\tau _{l+1}-\tau _{l}\leq \frac{1}{k}\text{ }\right\} \\
&\leq &P\left\{ \sup_{0<t\leq \frac{1}{k}}|w_{t}|\geq 2\varepsilon ^{2}\text{
}\right\} \leq \mathbb{P}\left\{ \sup_{0<t\leq \frac{1}{k}}|w_{t}^{1}|\geq
\varepsilon ^{2}\text{ }\right\} \\
&\leq &\exp \left( -\frac{\varepsilon ^{4}}{2}k\right) \text{.}
\end{eqnarray*}%
Therefore%
\begin{eqnarray*}
&&P\left\{ \sup_{l}\sup_{\tau _{l}\leq t\leq \tau _{l+1}}|w_{t}-\varepsilon 
\boldsymbol{n}_{l}|>3\sqrt{2}\varepsilon \right\} \\
&\leq &P\left\{ \sup_{l}\sup_{\tau _{l}\leq t\leq \tau
_{l+1}}|w_{t}-\varepsilon \boldsymbol{n}_{l}|>3\sqrt{2}\varepsilon :M\leq
k\right\} +P\left\{ M>k\right\} \\
&\leq &k\left( \frac{1}{3}\right) ^{\left[ \frac{1}{2\varepsilon }\right]
}+\exp \left( -\frac{\varepsilon ^{4}}{2}k\right)
\end{eqnarray*}%
by choosing $k=\frac{1}{\varepsilon ^{6}}$ to obtain%
\begin{eqnarray*}
&&P\left\{ \sup_{l}\sup_{\tau _{l}\leq t\leq \tau _{l+1}}\left(
|w_{t}-\varepsilon \boldsymbol{n}_{l}|\right) >3\sqrt{2}\varepsilon \right\}
\\
&\leq &\frac{1}{\varepsilon ^{6}}\left( \frac{1}{3}\right) ^{\left[ \frac{1}{%
2\varepsilon }\right] }+\exp \left( -\frac{1}{2\varepsilon ^{2}}\right)
\end{eqnarray*}%
so by the Borel-Cantelli lemma, $w(\varepsilon _{n})\rightarrow w$ almost
surely for a properly chosen $\varepsilon _{n}$ such that 
\begin{equation*}
\sum \frac{1}{\varepsilon _{n}^{6}}\left( \frac{1}{3}\right) ^{\left[ \frac{1%
}{2\varepsilon _{n}}\right] }+\exp \left( -\frac{1}{2\varepsilon _{n}^{2}}%
\right) <\infty \text{.}
\end{equation*}
\end{proof}

On the other hand the gap between two boxes $H_{\boldsymbol{z}}^{\varepsilon
}$ and $Z_{\boldsymbol{z}}^{\varepsilon }$ in comparison to the gap between $%
Z_{\boldsymbol{z}}^{\varepsilon }$ and $V_{\boldsymbol{z}}^{\varepsilon }$
is so small, it happens that $M_{H}=M_{Z}$ and $\boldsymbol{n}_{k}=%
\boldsymbol{m}_{k}$ on $\{k\leq M_{H}=M_{Z}\}$ with a large probability,
which is the context of the following lemma.

\begin{lemma}
\label{lemadd1}For any $\varepsilon \in (0,\frac{1}{4})$ we have%
\begin{equation}
P\{M_{H}=M_{Z}\text{ and }\boldsymbol{n}_{k}=\boldsymbol{m}_{k}\text{ for }%
k\leq M_{H}\}\geq \beta _{\varepsilon }\text{ }  \label{est-0q}
\end{equation}%
where $\beta _{\varepsilon }=1-2\varepsilon ^{4}-e^{-\frac{1}{2\varepsilon
^{2}}}$.
\end{lemma}

\begin{proof}
Let $A_{k}=\{\tau _{k}=\zeta _{k}$ and $\boldsymbol{n}_{k}=\boldsymbol{m}%
_{k}\}$ and $B_{k}=\cap _{l\leq k}A_{l}$. Then, as 
\begin{equation*}
\zeta _{k+1}\geq \zeta _{k}+T_{\partial V_{\boldsymbol{m}_{k}}^{\varepsilon
}}\circ \theta _{\zeta _{k}}\text{,}
\end{equation*}%
by strong Markov property and (\ref{cap-02}), $P\{B_{k+1}|B_{k}\}\geq
1-\varepsilon ^{10}$. Therefore%
\begin{equation*}
P\{B_{[\varepsilon ^{-6}]}\}\geq (1-\varepsilon ^{10})^{\varepsilon ^{6}}%
\text{.}
\end{equation*}%
Since $\varepsilon \in (0,\frac{1}{4})$ and $\log (1-x)\geq -2x$ for $x\in
(0,\frac{1}{2})$ we therefore have%
\begin{equation*}
\varepsilon ^{6}\log (1-\varepsilon ^{10})\geq -2\varepsilon ^{4}
\end{equation*}%
so that 
\begin{equation*}
P\{B_{[\varepsilon ^{-6}]}\}\geq e^{-2\varepsilon ^{4}}\geq 1-2\varepsilon
^{4}\text{.}
\end{equation*}%
On the other hand 
\begin{equation*}
P\{M_{H}>\varepsilon ^{6}\}\leq e^{-\frac{1}{2\varepsilon ^{2}}}
\end{equation*}%
so that%
\begin{eqnarray*}
P\{M_{H} &=&M_{Z}\text{ and }\boldsymbol{n}_{k}=\boldsymbol{m}_{k}\text{ for 
}k\leq M_{H}\} \\
&\geq &P\{B_{[\varepsilon ^{-6}]}\}-P\{M_{H}>\varepsilon ^{6}\} \\
&\geq &1-2\varepsilon ^{4}-e^{-\frac{1}{2\varepsilon ^{2}}}
\end{eqnarray*}%
which proves the lemma.
\end{proof}

\section{Proof of Theorem \protect\ref{th-main}: using the signatures}

This section is devoted to the proof of Theorem \ref{th-main} by using
information of its (extended) Stratonovich signatures. To this end, we need
to choose a good version of multiple iterated Stratonovich's integrals.

Recall that $(\boldsymbol{W},\mathcal{B},P)$ is the classical Wiener space,
where $\boldsymbol{W}$ is the sample space of all continuous paths started
at $0$, on which the coordinate process $(W_{t})_{t\geq 0}$ is Brownian
motion under probability measure $P$. For each path $w\in \boldsymbol{W}$,
and natural number $n$, we consider its dyadic approximations $w^{(n)}\in 
\boldsymbol{W}$ defined to be the polygon assuming the same values as $w$ at
dyadic points $\frac{j}{2^{n}}$ (for $j\in \mathbb{Z}_{+}$). According to
Wong-Zakai \cite{MR0183023} and Ikeda-Watanabe \cite{MR637061}, there is a
subset $\mathcal{N}\subset \boldsymbol{W}$ with probability zero, such that 
\begin{equation*}
\lim_{n\rightarrow \infty }\dint\limits_{s<t_{1}<\cdots <t_{k}<t}\alpha
^{1}(dw_{t_{1}}^{(n)})\cdots \alpha ^{k}(dw_{t_{k}}^{(n)})
\end{equation*}%
exists for every $w\in \boldsymbol{W}\setminus \mathcal{N}$, for all smooth
differential forms $\alpha ^{j}$ with bounded derivatives and for every pair 
$s<t$. The previous limit is denoted by $[\alpha ^{1}\cdots \alpha
^{n}](w)_{s,t}$. We fix such an exceptional set $\mathcal{N}$, and assign $%
[\alpha ^{1}\cdots \alpha ^{n}](w)$ to be zero for $w\in \mathcal{N}$. The
important fact is that $[\alpha ^{1}\cdots \alpha ^{n}]_{s,t}$ is a version
of Stratonovich's iterated integral%
\begin{equation*}
\int_{s<t_{1}<\cdots <t_{k}<t}\alpha ^{1}(\circ dW_{t_{1}})\cdots \alpha
^{k}(\circ dW_{t_{k}})\text{.}
\end{equation*}%
In Lyons and Qian \cite{MR2036784}, a specific exceptional set $\mathcal{N}$
was constructed by means of the so-called $p$-variation metric, which is
however not needed in our proof of the main theorem.

In this section $[\alpha ^{1}\cdots \alpha ^{n}]$ denotes the version of
Stratonovich's iterated integral $[\alpha ^{1}\cdots \alpha ^{n}]_{0,1}$
defined as above, so that $[\alpha ^{1}\cdots \alpha ^{n}]=0$ on $\mathcal{N}
$.

Our goal is to show that $W_{t}$ for all $t\leq 1$ is $\mathcal{G}_{1}$%
-measurable. For $\varepsilon \in (0,1/4)$, and choose $\alpha $ and $\beta $
big enough so that the estimates in Lemmata \ref{lemc1} and \ref{lemc2}
hold. Choose a smooth 1-form on $\mathbb{R}^{2}$, $\phi
(x_{1},x_{2})=f(x_{1},x_{2})dx_{1}$, with a compact support in $Z_{%
\boldsymbol{0}}^{\varepsilon }$ such that $f(x_{1},x_{2})=x_{2}$ on $K_{%
\boldsymbol{0}}^{\varepsilon }$. For each $\boldsymbol{z}\in \mathbb{Z}^{2}$%
, let $\phi ^{\boldsymbol{z}}=\phi (\cdot -\varepsilon \boldsymbol{z})$ (or $%
\phi ^{\boldsymbol{z},\varepsilon }$ if we wish to indicate the dependence
on $\varepsilon $) be the translation of $\phi $ with compact support in $Z_{%
\boldsymbol{z}}^{\varepsilon }$. Therefore, $\{\phi ^{\boldsymbol{z}}:%
\boldsymbol{z}\in \mathbb{Z}^{2}$, $\varepsilon \in (0,1/4)\}$ is a
countable family of non-trivial differential forms with disjoint compacts
for every fixed $\varepsilon $. The key idea, as we have explained in the
Introduction, is to read out the blocks $Z_{\boldsymbol{n}_{l}}^{\varepsilon
}$'s which have been visited by the Brownian motion by using the extended
Stratonovich's signatures of form \ $[\phi ^{\boldsymbol{z}_{1}}\cdots \phi
^{\boldsymbol{z}_{m}}]$.

Let $m\geq 0$. A finite ordered sequence (or called a word) of length $m+1$, 
$\langle \boldsymbol{z}_{0}\cdots \boldsymbol{z}_{m}\rangle $ (where all $%
\boldsymbol{z}$'s belong to the lattice $\mathbb{Z}^{2}$), is admissible if $%
\boldsymbol{z}_{l}\neq \boldsymbol{z}_{l+1}$ for $l=0,\cdots ,m-1$. Let $%
\mathcal{W}_{m}$ denote the set of all admissible words of length $m+1$.

If $w\in \boldsymbol{W}$, 
\begin{equation*}
\hat{M}(w)=\sup \left\{ m:[\phi ^{\boldsymbol{z}_{0}}\cdots \phi ^{%
\boldsymbol{z}_{m}}](w)\neq 0\text{ for some }\langle \boldsymbol{z}%
_{0}\cdots \boldsymbol{z}_{m}\rangle \in \mathcal{W}_{m}\right\}
\end{equation*}%
so that $\hat{M}$ is $\mathcal{G}_{1}$-measurable. For each $m\in \mathbb{N}$
and each admissible word $\langle \boldsymbol{z}_{0}\cdots \boldsymbol{z}%
_{m}\rangle \in \mathcal{W}_{m}$ define%
\begin{equation}
\boldsymbol{A}_{m,\langle \boldsymbol{z}_{0}\cdots \boldsymbol{z}_{m}\rangle
}=\{\hat{M}(w)=m\text{ and }\left[ \phi ^{\boldsymbol{z}_{0}}\cdots \phi ^{%
\boldsymbol{z}_{m}}\right] (w)\neq 0\}\text{.}  \label{dec-02}
\end{equation}

Since $\phi ^{\boldsymbol{z}}$ have disjoint supports, therefore, if $\zeta
_{m+1}(w)>1$, then $\hat{M}(w)$ can not be greater than $m$, so that $\hat{M}%
\leq M_{Z}$ except on the exceptional set $\mathcal{N}$. On the other hand,
according to Lemma \ref{lem-key1} and the strong Markov property, $\hat{M}%
\geq M_{H}$ almost surely. Therefore $M_{H}\leq \hat{M}\leq M_{Z}$ almost
surely.

If $\hat{M}(w)=m$, there is at most one $\langle \boldsymbol{z}_{0}\cdots 
\boldsymbol{z}_{m}\rangle \in \mathcal{W}_{m}$ such that $\left[ \phi ^{%
\boldsymbol{z}_{0}}\cdots \phi ^{\boldsymbol{z}_{m}}\right] (w)\neq 0$ and
all other $[\phi ^{\boldsymbol{z}_{0}^{\prime }}\cdots \phi ^{\boldsymbol{z}%
_{n}^{\prime }}](w)=0$ for $\langle \boldsymbol{z}_{0}^{\prime }\cdots 
\boldsymbol{z}_{n}^{\prime }\rangle \in \mathcal{W}_{n}$ if $n>m$ or if $n=m$
but $\langle \boldsymbol{z}_{0}^{\prime }\cdots \boldsymbol{z}_{m}^{\prime
}\rangle \neq \langle \boldsymbol{z}_{0}\cdots \boldsymbol{z}_{m}\rangle $.

Let%
\begin{equation}
\boldsymbol{\tilde{W}}_{m,\langle \boldsymbol{z}_{0}\cdots \boldsymbol{z}%
_{m}\rangle }=\{M_{H}=m\text{, }\boldsymbol{n}_{l}=\boldsymbol{z}_{l}\text{
for }l=0,\cdots ,m\}\text{.}  \label{dec-01}
\end{equation}%
for each admissible word $\langle \boldsymbol{z}_{0}\cdots \boldsymbol{z}%
_{m}\rangle \in \mathcal{W}_{m}$, and%
\begin{equation*}
\boldsymbol{\tilde{W}}_{\varepsilon }=\dbigcap\limits_{m=0}^{\infty
}\dbigcup\limits_{\langle \boldsymbol{z}_{0}\cdots \boldsymbol{z}_{m}\rangle
\in \mathcal{W}_{m}}\boldsymbol{\tilde{W}}_{m,\langle \boldsymbol{z}%
_{0}\cdots \boldsymbol{z}_{m}\rangle }\text{.}
\end{equation*}%
Then, according to Lemma \ref{est-0q}, $P(\boldsymbol{\tilde{W}}%
_{\varepsilon })\geq \beta _{\varepsilon }$.

We are now in a position to complete our proof. Set 
\begin{equation*}
\boldsymbol{\tilde{n}}_{l}=\sum_{m=0}^{\infty }\sum_{\langle \boldsymbol{z}%
_{0}\cdots \boldsymbol{z}_{m}\rangle \in \mathcal{W}_{m}}\boldsymbol{z}%
_{l}1_{\boldsymbol{A}_{m,\langle \boldsymbol{z}_{0}\cdots \boldsymbol{z}%
_{m}\rangle }}
\end{equation*}%
and redefine 
\begin{equation*}
\hat{w}(\varepsilon )_{t}=\boldsymbol{\tilde{n}}_{l}\varepsilon +\frac{%
t-\tau _{l}}{\tau _{l+1}-\tau _{l}}\boldsymbol{\tilde{n}}_{l+1}\varepsilon 
\text{ \ \ \ if }t\in \lbrack \tau _{l},\tau _{l+1}]
\end{equation*}%
then, we may choose a sequence $\varepsilon _{n}\downarrow 0$ so that $%
\sum_{n}(1-\beta _{\varepsilon _{n}})<\infty $. Then, $\hat{w}(\varepsilon
_{n})=w(\varepsilon _{n})$ almost surely on $\boldsymbol{\tilde{W}}%
_{\varepsilon }$. Since $P(\boldsymbol{\tilde{W}}_{\varepsilon })\geq \beta
_{\varepsilon }$, it follows the Borel-Cantelli lemma, $\sup_{t\in \lbrack
0,1]}|\hat{w}(\varepsilon _{n})-w(\varepsilon _{n})|\rightarrow 0$ in
probability as $n\rightarrow \infty $, and therefore $W_{t}\in \mathcal{G}%
_{1}$ for $t\leq 1$.

\bigskip

\bibliographystyle{amsplain}
\bibliography{signature-2011-Feb2-yves}

\noindent {\small \textsc{Yves Le Jan}, D\'{e}partment de Math\'{e}matiques, Universit\'{e} Paris-Sud 11, 91405 Orsay, France. Email: \texttt{yves.lejan@math.u-psud.fr}}

\vskip0.1truecm

\noindent {\small \textsc{Zhongmin Qian}, Mathematical Institute,University of Oxford, Oxford OX1 3LB, England. Email: \texttt{qianz@maths.ox.ac.uk}}

\end{document}